\documentclass[final]{elsarticle}
\usepackage{hyperref}
\usepackage[textwidth=16cm, textheight=24cm, left=2.5cm, right=2.5cm, top=3cm, bottom=3cm]{geometry}
\usepackage{latexsym}
\usepackage{enumerate}
\usepackage{amsmath, amsthm}
\usepackage{amssymb}
\usepackage{graphics}
\usepackage{graphicx}
\usepackage{subfigure}
\usepackage{indentfirst}
\usepackage{mathrsfs}
\usepackage{tikz}
\usetikzlibrary{decorations.pathreplacing}

\theoremstyle{plain}
\newtheorem{thm}{Theorem}[section]

\newtheorem{lem}{Lemma}[section]

\theoremstyle{definition}
\newtheorem{defn}{Definition}[section]
\numberwithin{equation}{section}

\makeatletter
\def\ps@pprintTitle{%
  \let\@oddhead\@empty
  \let\@evenhead\@empty
  \def\@oddfoot{\reset@font\hfil\thepage\hfil}
  \let\@evenfoot\@oddfoot
}
\makeatother

\begin{document}
\begin{frontmatter}
\title{A series of trees with the first  $\lfloor\frac{n-7}{2}\rfloor$ largest  energies\tnoteref{title1}}
\tnotetext[title1]{Supported by NSF of China No.10731040,  No.11101088, No.11026147 and No.11101263}
\author[a]{Hai-Ying Shan}
\ead{shan\_haiying@tongji.edu.cn}
\author[a]{Jia-Yu Shao\corref{cor1}}
\ead{jyshao@tongji.edu.cn}
\author[a]{Li Zhang}
\ead{lizhang@tongji.edu.cn}
\author[b]{Chang-Xiang He}
\ead{changxianghe@hotmail.com}
\cortext[cor1]{Corresponding author.}

\address[a]{Department of Mathematics,  Tongji University,  Shanghai  200092,  China}
\address[b]{College of Science, University of Shanghai for Science and Technology, Shanghai, 200093, China}
\date{}
\begin{abstract}

The energy of a graph is defined as the sum of the absolute values of the eigenvalues of the graph.  In this paper,  we present a new method to compare the energies  of two $k$-subdivision bipartite graphs on some cut edges.  As the applications of this new method, we determine the first $\lfloor\frac{n-7}{2}\rfloor$ largest  energy trees of order $n$ for $n\ge 31$, and we also give a simplified proof of the conjecture on the fourth maximal energy tree.

\end{abstract}

\end{frontmatter}

\section{Introduction}

Let $G$ be a graph with $n$ vertices and $A$ be its adjacency
matrix. Let $\lambda_{1}, \cdots, \lambda_{n}$ be the eigenvalues of
$A$,  then the $energy$ of $G$,  denoted by $\mathbb{E}(G)$,  is  defined
\cite{gutman1978eg, gutman2001ego} as
$\mathbb{E}(G)=\displaystyle\sum_{i=1}^{n}|\lambda_{i}|$.

The characteristic polynomial $\det(x I-A)$ of the adjacency matrix
$A$ of a graph $G$ is also called the characteristic polynomial of
$G$,  written as $\phi(G, x) = \sum\limits_{i=0}^na_i(G)x^{n-i}$.

In this paper,  we write  $b_i(G)= |a_i(G)|$,  and also write $$\widetilde{\phi}(G, x) = \sum\limits_{i=0}^n b_i(G)x^{n-i}.$$

If $G$ is a bipartite graph,  then it is well known that $\phi(G, x)$ has the form

\begin{equation}
\phi(G, x) = \sum\limits_{i=0}^{\lfloor \frac{n}{2} \rfloor} a_{2i}(G)x^{n-2i}=\sum\limits_{i=0}^{\lfloor \frac{n}{2} \rfloor} (-1)^{i}b_{2i}(G)x^{n-2i}\label{equ1.1}
\end{equation}

\noindent and thus
\begin{equation}\label{equ1.2}
\widetilde{\phi}(G, x) = \sum\limits_{i=0}^{\lfloor \frac{n}{2} \rfloor} b_{2i}(G)x^{n-2i}.\qquad(b_{2i}(G)= |a_{2i}(G)|=(-1)^{i}a_{2i})
\end{equation}
In case $G$ is a forest,  then $b_{2i}(G)= m(G, i)$,  the number of $i$-matchings of $G$.

The following integral formula by Gutman and Polansky (\cite{gutman1986mathematical}) on the difference of the energies of two graphs is the starting point of this paper.

\begin{equation}\label{equ1.3}
\mathbb{E}(G_1)-\mathbb{E}(G_2)=\frac{1}{\pi}\int\limits_{-\infty}^{+\infty}\ln \left|\frac{\phi(G_1, ix)}{\phi(G_2, ix)}\right|\text{d} x \qquad \qquad (i=\sqrt{-1})
\end{equation}

Now suppose again that  $G$ is a  bipartite graph of order $n$. Then  by (\ref{equ1.1}) and (\ref{equ1.2}) we have
\begin{equation}\label{equ1.4}
\phi(G, i x)=i^{n}\widetilde{\phi}(G, x) \text{\qquad\qquad ($G$ is bipartite, $i=\sqrt{-1}$) }
\end{equation}

Using (\ref{equ1.4}) we can derive the following new formula from (\ref{equ1.3}) which does not involve the complex number \nolinebreak  $i$.

 \begin{thm}\label{thm1.1}
 If $G_1, G_2$ are both bipartite graphs of order $n$,  then we have
 \begin{equation}\label{equ1.5}
 \mathbb{E}(G_1)-\mathbb{E}(G_2)=\frac{2}{\pi}\int\limits_{0}^{+\infty}\ln \frac{\widetilde{\phi}(G_1, x)}{\widetilde{\phi}(G_2, x)}{\rm{d}} x
 \end{equation}
  \end{thm}
\begin{proof}
Since $G_1, G_2$ are both bipartite graphs of order $n$,  it is easy to see that

$\displaystyle\frac{\widetilde{\phi}(G_1, x)}{\widetilde{\phi}(G_2, x)}=
\frac{\sum\limits_{j=0}^{\lfloor \frac{n}{2} \rfloor} b_{2j}(G_{1})x^{n-2j}}{\sum\limits_{j=0}^{\lfloor \frac{n}{2} \rfloor} b_{2j}(G_{2})x^{n-2j}}
$ is an even function and $\displaystyle\frac{\widetilde{\phi}(G_1, x)}{\widetilde{\phi}(G_2, x)}>0$ for $x >0$.\\ So from (\ref{equ1.3}) and (\ref{equ1.4}) we have
$$\begin{aligned}
\mathbb{E}(G_1)-\mathbb{E}(G_2)=&\frac{1}{\pi}\int\limits_{-\infty}^{+\infty}\ln \left|\frac{\phi(G_1, ix)}{\phi(G_2, ix)}\right|\text{d} x
=\frac{1}{\pi}\int\limits_{-\infty}^{+\infty}\ln \left|\frac{\widetilde{\phi}(G_1, x)}{\widetilde{\phi}(G_2, x)}\right|\text{d} x
=\frac{2}{\pi}\int\limits_{0}^{+\infty}\ln \frac{\widetilde{\phi}(G_1, x)}{\widetilde{\phi}(G_2, x)}\text{d} x.
\end{aligned}$$\end{proof}

%
\begin{defn}
Let $f(x) = \sum\limits_{i=0}^n a_i x^{n-i}$ and  $g(x) = \sum\limits_{i=0}^n b_i x^{n-i}$ be two monic polynomials of degree $n$ with nonnegative coefficients.
 \begin{enumerate}[(1).]
 \item If $a_{i} \leq b_{i}$ for all $0 \leq i \leq n$,  then we write $f(x) \preccurlyeq g(x)$.
 \item If $f(x) \preccurlyeq g(x)$ and $f(x) \neq g(x)$,  then we write $f(x) \prec g(x)$.
 \end{enumerate}
\end{defn}

Now we define the following quasi-order for bipartite graphs (which is equivalent to the well known  quasi-order defined by the coefficients $b_{i}(G)$ ).

\begin{defn}
Let $G_{1}$ and $G_{2}$ be two bipartite graphs of order $n$. Then we write $G_{1} \preccurlyeq G_{2}$ if $\widetilde{\phi}(G_{1}, x)\preccurlyeq \widetilde{\phi}(G_{2}, x)$,  write $G_{1} \prec G_{2}$ if $\widetilde{\phi}(G_{1}, x)\prec \widetilde{\phi}(G_{2}, x)$ and write $G_{1} \sim G_{2}$ if $\widetilde{\phi}(G_{1}, x) = \widetilde{\phi}(G_{2}, x).$
\end{defn}

According to the integral formula in Theorem \ref{thm1.1},  we can see that  for two bipartite graphs $G_{1}$ and $G_{2}$ of order $n$,
$$G_1 \preccurlyeq G_2 \Longrightarrow \mathbb{E}(G_1) \le \mathbb{E}(G_2);\qquad\qquad and \qquad\qquad  G_1 \prec G_2\Longrightarrow \mathbb{E}(G_1)< \mathbb{E}(G_2).$$

The method of the quasi-order relation ``$\preccurlyeq$'' is an important tool in the study of graph energy.

Graphs with extremal energies are extensively studied in literature. Gutman \cite{gutman1977ase} determined the first and second maximal energy trees of order $n$;
N.Li,  S.Li \cite{li2008eet} determined the third maximal energy tree;
Gutman et al. \cite{gutman2008eet} conjectured that the fourth maximal energy tree is ~$P_{n}(2, 6, n-9)$ (see Fig.\ref{figstarlike} for this graph); B. Huo et al. \cite{huo2011complete} proved that this conjecture is true.

 In this paper, we first consider in \S 2 some recurrence relation of the polynomials $\widetilde{\phi}(G(k), x)$ for the $k$-subdivision graph $G(k)$ (on some cut edge $e$ of a bipartite graph $G$). Then in \S 3 we present a new method of directly comparing the energies  of two $k$-subdivision bipartite graphs $G(k)$ and $H(k)$ if they are quasi-order incomparable.
   Using this new method,  we are able to  provide a simplified proof of the above mentioned conjecture on the fourth maximal energy tree. The main result of this paper is that, we determine (in \S 5) the first $\lfloor\frac{n-7}{2}\rfloor$  largest energy trees of order $n\ge 31$ by using the new method of comparing energies given in \S 3. For example when $n\ge 2007$, we can determine the first 1000 largest energy trees of order n (but up to now, only the first four are known).

\section{Some recurrence relations of $\phi(G, x)$  and $\widetilde{\phi}(G, x)$  for $k$-subdivision bipartite graphs}
The following lemma is an alternative form of Heilbronner's recurrence formula \cite{heilrecursion}.
\begin{lem}\cite{heilrecursion}\label{lem2.1}   Let  $uv$ be a cut edge of a graph $G$,  then $\phi(G, x)=\phi(G-uv, x)-\phi(G-u-v, x).$
\end{lem}
\begin{figure}[h]
\begin{center}
\begin{tabular}{ccc}
\begin{tikzpicture}
\tikzstyle myline=[line width=0.8pt]
\coordinate (A) at (-0.7, 0);
\coordinate (B) at (0.7, 0) ;
\draw (A)--(B)  node [midway,  above=1pt] {$e$};
\node at (-1.4, 0) {$G_{1}$};
\node at (1.64, 0) {$G_{2}$};
\draw (-1.2, 0) circle (20pt);
\draw (1.2, 0) circle (20pt);
\foreach \point in {A, B}
{\fill [black] (\point) circle (1.5pt);}
\node[below] at (A) {$u$};
\node[below] at (B) {$v$};
\node at (0, -1) {$H_{1}$};
\end{tikzpicture}\hspace{1cm} &
\begin{tikzpicture}
\tikzstyle myline=[line width=0.8pt]
\coordinate (A) at (-0.9, 0);
\coordinate (B) at (0.9, 0) ;
\coordinate (C) at (0, 0) ;
\draw (A)--(C)  node [midway,  above=1pt] {$e_{1}$};
\draw (C)--(B)  node [midway,  above=1pt] {$e_{2}$};
\node at (-1.7, 0) {$G_{1}$};
\node at (1.7, 0) {$G_{2}$};
\draw (-1.4, 0) circle (20pt);
\draw (1.4, 0) circle (20pt);
\foreach \point in {A, B, C}
{\fill [black] (\point) circle (1.5pt);}
\node[left] at (A) {$u$};
\node[right] at (B) {$v$};
\node[below] at (C) {$w$};
\node at (0, -1) {$H_{2}$};
\end{tikzpicture}\hspace{1cm} &

\begin{tikzpicture}
\tikzstyle myline=[line width=0.8pt]
\coordinate (A) at (-1.35, 0);
\coordinate (B) at (-0.45, 0) ;
\coordinate (C) at (0.45, 0) ;
\coordinate (D) at (1.35, 0) ;
\draw (A)--(B) ;
\draw (B)--(C)  node [midway,  above=1pt] {$e'$};
\draw (C)--(D);
\node at (-2.15, 0) {$G_{1}$};
\node at (2.15, 0) {$G_{2}$};
\draw (-1.85, 0) circle (20pt);
\draw (1.85, 0) circle (20pt);
\foreach \point in {A, B, C,D}
{\fill [black] (\point) circle (1.5pt);}
\node[left] at (A) {$u$};
\node[right] at (D) {$v$};
\node[below] at (B) {$w$};
\node[below] at (C) {$y$};
\node at (0, -1) {$H_{3}$};
\end{tikzpicture}
\end{tabular}\\
\end{center}
\vspace*{-0.5cm}\caption{The graphs $H_{1}, H_{2}$ and $H_{3}$ }\label{fig2}
\end{figure}
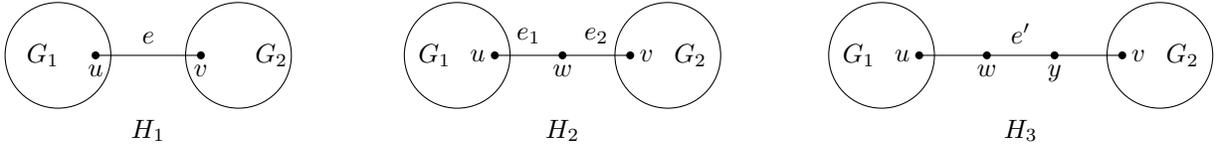

For the sake of simplicity, we sometime abbreviate $\phi(G,x)$ by $\phi(G)$.

The following relation can be derived from Lemma \ref{lem2.1}.
\begin{lem}\label{lem2.2}
Let $H_{1}, H_{2}, H_{3}$ be graphs as shown in  Fig.\ref{fig2}.  Then we have
$$\phi(H_{3}, x)=x \phi(H_{2}, x)-\phi(H_{1}, x)$$
\end{lem}
\begin{proof} Let $G_1'$ be the graph obtained from $G_1$ by attaching a new pendent edge $uw$ to $G_1$ at $u$, and $G_2'$ be the graph obtained from $G_2$ by attaching a new pendent edge $vy$ to $G_2$ at $v$. Then by using  Lemma \ref{lem2.1} we have
$$\phi(G'_{1})= x \phi(G_{1})- \phi (G_{1}-u), \quad \mbox {and} \quad \phi(G'_{2}) = x \phi(G_{2})- \phi (G_{2}-v).$$
Now using Lemma \ref{lem2.1} for $H_3$ and its cut edge $e'=wy$,  we have
$$\begin{aligned} &\phi(H_{3})=\phi(H_{3}-e')- \phi(H_{3}-w-y) = \phi(G'_{1}) \phi(G'_{2})- \phi(G_{1}) \phi(G_{2}) \\
=&(x \phi(G_{1})- \phi (G_{1}-u))( x \phi(G_{2})- \phi (G_{2}-v))- \phi(G_{1}) \phi(G_{2})\\
=& (x^{2}-1)\phi(G_{1}) \phi(G_{2})-x\phi(G_{1})\phi (G_{2}-v)-x\phi(G_{2})\phi (G_{1}-u)+ \phi(G_{1}-u) \phi(G_{2}-v)\end{aligned}$$

\noindent Also using Lemma \ref{lem2.1}  for $H_{2}$ and $H_{2}-e_{1}$ we have£º
$$\begin{aligned}
\phi(H_{2})=&\phi(H_{2}-e_{1})- \phi(H_{2}-u-w)
=\phi(H_{2}-e_{1}-e_{2})-\phi(H_{2}-e_{1}-w-v)- \phi((G_{1}-u )\cup G_{2})\\
=& x \phi(G_{1}) \phi (G_{2}) - \phi(G_{1}) \phi(G_{2}-v) -\phi(G_{1}-u) \phi(G_{2})
\end{aligned}$$

\noindent Using Lemma \ref{lem2.1} for $H_{1}$ we also have
$$\phi(H_{1})=\phi(H_{1}-e)- \phi(H_{1}-u-v)= \phi(G_{1}) \phi (G_{2}) - \phi(G_{1}-u) \phi(G_{2}-v)$$

\noindent Now it is easy to verify from the above three equations that $\phi(H_{3})=x \phi(H_{2})-\phi(H_{1})$.

\end{proof}

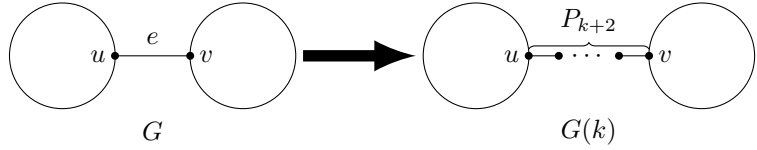
\begin{figure}[h]
\begin{center}
\begin{tikzpicture}
\begin{scope}[xshift=-1.5cm]
\tikzstyle myline=[line width=0.8pt]
\coordinate (A) at (-0.5, 0);
\coordinate (B) at (0.5, 0) ;
\draw (A)--(B)  node [midway,  above=1pt] {$e$};
\draw (-1.2, 0) circle (20pt);
\draw (1.2, 0) circle (20pt);
\foreach \point in {A, B}
{\fill [black] (\point) circle (1.5pt);}
\node[left] at (A) {$u$};
\node[right] at (B) {$v$};
\node at (0, -1) {$G$};
\end{scope}
\begin{scope}[xshift=0.5cm]
\pgfsetarrowsend{latex}
\pgfsetlinewidth{1ex}
\pgfpathmoveto{\pgfpoint{0cm}{0cm}}
\pgfpathlineto{\pgfpoint{1.5cm}{0cm}}
\pgfusepath{stroke}
\useasboundingbox (-0.25, -0.25) rectangle (3.75, 2.25);
\end{scope}
\begin{scope}[xshift=4.3cm]
\tikzstyle myline=[line width=0.8pt]
\coordinate (A) at (-0.8, 0);
\coordinate (B) at (-0.4, 0) ;
\coordinate (C) at (0.4, 0) ;
\coordinate (D) at (0.8, 0) ;
\draw (A)--(B) ;
\draw (C)--(D) ;
\node at (0, 0) {$\cdots$};
\draw (-1.5, 0) circle (20pt);
\draw (1.5, 0) circle (20pt);
\foreach \point in {A, B, C, D}
{\fill [black] (\point) circle (1.5pt);}
\node[left] at (A) {$u$};
\node[right] at (D) {$v$};
\node at (0, -1) {$G(k)$};
\draw [decoration={brace, raise=2.5pt},  decorate] (A) -- (D)node [midway,  above=5pt]{$P_{k+2}$};
\end{scope}
\end{tikzpicture}
\end{center}
    \vspace*{-0.8cm}\caption{Graph $G$ and its $k$-subdivision graph}
\end{figure}

\begin{defn}
Let $e$ be a cut edge of a graph $G$,  and $G_{e}(k)$ denote the graph obtained by replacing $e$ with a path of length $k+1$ (for simplicity of notations, we usually we abbreviate $G_{e}(k)$ by $G(k)$ ). We say that $G(k)$ is a  $k$-subdivision graph of $G$ on the cut edge $e$. We also agree that $G(0)=G$.
\end{defn}

From Lemma \ref{lem2.2},  we have the following recurrence relation for $\phi(G(k), x)$.
\begin{thm}\label{thm2.1}
Let $G(k)$ be a $k$-subdivision graph of $G$ on the cut edge $e$ of $G$,  then we have $$\phi(G(k+2), x)=x \phi(G(k+1), x)-\phi(G(k), x) \qquad\qquad (k \geq 0)$$
\end{thm}

\begin{proof}
Take $H_1=G(k)$ in Lemma \ref{lem2.2} and $e$ be an edge in $H_{1}$ on the path of length $k+1$ obtained by $k$-subdividing the edge $e$. Then $H_2=G(k+1)$ and $H_3=G(k+2)$. The result now follows from Lemma \ref{lem2.2}.
\end{proof}

\begin{thm}\label{thm2.2}
Let $G$ be a bipartite graph of order $n$ and $G(k)$ be a $k$-subdivision graph (of order $n+k$) of $G$ on some cut edge $e$. Then we have
 \begin{equation}\label{equ2.1}
 \widetilde{\phi}(G(k+2), x)=x \widetilde{\phi}(G(k+1), x)+ \widetilde{\phi}(G(k), x) \qquad\qquad (k \geq 0)
 \end{equation}
\end{thm}
\begin{proof}
By Theorem \ref{thm2.1},  we have
$$\phi(G(k+2), x)= x \phi(G(k+1), x)-\phi(G(k), x)$$
substitute $x$ by $i x$,  we get
$$\phi(G(k+2), ix)=i x \phi(G(k+1), ix)-\phi(G(k), ix).$$
Now using (\ref{equ1.4}) for $G(k+2), G(k+1)$ and $G(k)$ (since they are all bipartite) we
have $$
i^{n+k+2}\widetilde{\phi}(G(k+2), x)= i^{n+k+1}\dot i x\widetilde{\phi}(G(k+1), x)
-i^{n+k}\widetilde{\phi}(G(k), x)
$$
Dividing both sides by $i^{n+k+2}$ we get (\ref{equ2.1}).
\end{proof}

\begin{thm}\label{thm2.3}
Let $e,  e'$ be cut edges of bipartite graphs $G$ and $H$ of order $n$ ,  respectively.  If $G(0)\preccurlyeq H(0)$ and  $G(1)\preccurlyeq H(1)$,  then
 we have $G(k)\preccurlyeq H(k)$ for all $k \ge 2$,
with $G(k)\sim H(k)$ if and only if both the
  two relations  $H(0)\sim G(0)$ and $H(1) \sim G(1)$ hold.
\end{thm}
\begin{proof}
The result follows directly from Theorem \ref{thm2.2} and induction on $k$.
\end{proof}

\begin{thm}\label{thm2.4}
Let $G,  H$ be  bipartite  graphs  of order $n$,  $e_{1},  e_{2}$ be two cut edges of $G$ and   ${e}_{1}',  {e}_{2}'$ be two cut edges of $H$. Let $G(a, b)$ denote the graph obtained from $G$ by subdividing  $e_{1},  e_{2}$ by $a, b$ times, and $H(c, d)$ denote the graph obtained from $H$ by subdividing  ${e}_{1}',  {e}_{2}'$ by $c, d$ times, respectively.
If
\begin{align}
G(0, 0)\preccurlyeq H(0, 0)  \qquad\qquad and \qquad\qquad  &G(0, 1)\preccurlyeq H(0, 1), \label{equ2.2}\\
G(1, 0)\preccurlyeq H(1, 0)  \qquad\qquad and \qquad\qquad  &G(1, 1)\preccurlyeq H(1, 1)\label{equ2.3}
\end{align}
then we have $G(l, k)\preccurlyeq H(l, k)$ for all $l\geq 0$ and $k\geq 0$. Moreover, if one of $l$ and $k$ is at least 2, then $G(l, k)\prec  H(l, k) $ if each of (\ref{equ2.2}) and (\ref{equ2.3}) contains at least one strict relation.
\end{thm}
\begin{proof}
Using  Theorem \ref{thm2.3} for $e_{2}$ and $e_{2}'$ we have
\begin{align}
\mbox {(\ref{equ2.2})}  &\Longrightarrow G(0, k)\preccurlyeq H(0, k) \quad (k \ge 0) , \label{equ2.4}\\
\mbox {(\ref{equ2.3})}  &\Longrightarrow G(1, k)\preccurlyeq H(1, k) \quad (k \ge 0)\label{equ2.5}.
\end{align}

\noindent Now using  Theorem \ref{thm2.3} for $e_{1}$ and $e_{1}'$ we also have

\hspace {3cm} (2.4) and (2.5)   $\Longrightarrow  G(l, k)\preccurlyeq H(l, k)$ \quad ($l \ge 0$).

When   (\ref{equ2.2}) and (\ref{equ2.3}) both contain strict relations,  we have both strict relations in  (\ref{equ2.4}) and (\ref{equ2.5}) for $k\geq 2$. Thus   $G(l, k)\prec  H(l, k) $ for all $k\geq 2$  by Theorem \ref{thm2.3}. Similar arguments apply to the case $l\geq 2$.
\end{proof}

\section{A new method of directly comparing the energies of $k$-subdivision  bipartite graphs}
Notice that if the conditions in Theorem \ref{thm2.3} do not hold, then $G(k)$ and $H(k)$ might be  quasi-order incomparable.
 In this section,  we present a new method to directly compare the energies  of two $k$-subdivision bipartite graphs $G(k)$ and $H(k)$ when they are quasi-order incomparable.
   Using this method,  we  give  a simplified proof of the conjecture on the fourth maximal energy tree.

In the following, we always write $g_{k}=\widetilde{\phi}(G(k), x)$,  $h_{k}=\widetilde{\phi}(H(k), x)$, and $d_{k}=\displaystyle\frac{h_{k}}{g_{k}}$.
\begin{lem}\label{lem3.1}
Let $G(k)$, $H(k)$ be $k$-subdivision graphs on some cut edges  of  the bipartite graphs $G$ and $H$ of order $n$,  respectively ($k\ge 0$), $g_k, h_k$ and $d_k$ be defined as above.
Then for each fixed $x>0$,  we have
\begin{enumerate}[(1).]
\item If $d_{1} > d_{0}$,  then $d_{0} < d_{k} <d_{1}$ for all $k \geq 2;$
\item If $d_{1} < d_{0}$,  then $d_{1} < d_{k} <d_{0}$ for all $k \geq 2;$
\item If $d_{1} = d_{0}$,  then $ d_{k} =d_{0}$ for all $k$.
\end{enumerate}
\end{lem}

\noindent (So in any case we have $d_{k} \geq \min\{d_{0},d_{1}\}$.)

\begin{proof}
By the recurrence relations in Theorem \ref{thm2.2}, we have
$$
\begin{aligned}
\displaystyle d_{k}=&\frac{h_{k}}{g_{k}}=\frac{x h_{k-1} + h_{k-2}}{x g_{k-1} + g_{k-2}}
=\frac{x d_{k-1}g_{k-1} +  d_{k-2} g_{k-2}}{x g_{k-1} +  g_{k-2}}\\
=&\left(\frac{x g_{k-1}}{x  g_{k-1} +  g_{k-2}}\right) d_{k-1}+
\left(\frac{  g_{k-2}}{x g_{k-1} +  g_{k-2}}\right) d_{k-2}
\end{aligned}$$

This tells us that $d_{k}$ is a convex combination of $d_{k-1}$ and $d_{k-2}$ with positive coefficients,  which implies that $d_{k}$ lies in the open interval $(d_{k-1},  d_{k-2})$ or $(d_{k-2},  d_{k-1})$ if $d_{k-1} \neq d_{k-2}$. Using this fact and the induction on $k$ we obtain that $d_{k}$ always lies in the open interval $(d_{0},  d_{1})$ or $(d_{1},  d_{0})$ when $d_{0} \neq d_{1}$, and $d_k=d_0$ when $d_1=d_0$.
\end{proof}

The following theorem can be derived  from Lemma \ref{lem3.1}:

\begin{thm}\label{thm3.1}

\begin{enumerate}[(1).]
\item
 If $h_{1}g_{0}-h_{0}g_{1}=\widetilde{\phi}(H(1), x)\widetilde{\phi}(G(0), x)-\widetilde{\phi}(H(0), x)\widetilde{\phi}(G(1), x)>0$ (which is equivalent to $d_{1}(x)>d_{0}(x)$) for all $x >0 $, then we have
$$\mathbb{E}(H(k))- \mathbb{E}(G(k))> \mathbb{E}(H(0))- \mathbb{E}(G(0)) \quad \text{ (for all $k > 0$.)}$$
\item If $h_{1}g_{0}-h_{0}g_{1}=\widetilde{\phi}(H(1), x)\widetilde{\phi}(G(0), x)-\widetilde{\phi}(H(0), x)\widetilde{\phi}(G(1), x)<0$(which is equivalent to $d_{1}(x)< d_{0}(x)$) for all $x >0 $,  then we have
    $$\mathbb{E}(H(k))- \mathbb{E}(G(k))> \mathbb{E}(H(1))- \mathbb{E}(G(1))\quad \text{ for all $k \neq 1$.}$$
\end{enumerate}
\end{thm}

\begin{proof}
(1). Since $d_{1}(x)>d_{0}(x)$ for all $x >0 $, by (1) of  Lemma \ref{lem3.1} we have $d_{k}(x) >d_{0}(x)$ for all $x>0$ and $k>0$.  So by (\ref{equ1.5}) we have

$\begin{aligned}
\mathbb{E}(H(k))- \mathbb{E}(G(k))&=\frac{2}{\pi}\int\limits_{0}^{+\infty}\ln\frac{\widetilde{\phi}(H(k), x)}{\widetilde{\phi}(G(k), x)}\text{d} x=\frac{2}{\pi}\int\limits_{0}^{+\infty}\ln d_{k}(x)\text{d} x\\
&>\frac{2}{\pi}\int\limits_{0}^{+\infty}\ln d_{0}(x)\text{d} x =\frac{2}{\pi}\int\limits_{0}^{+\infty}\ln\frac{\widetilde{\phi}(H(0), x)}{\widetilde{\phi}(G(0), x)}\text{d} x =\mathbb{E}(H(0))- \mathbb{E}(G(0))\quad (k>0).\\
\end{aligned}$

\noindent The proof of (2) is similar to that of (1).
\end{proof}

In \cite{shan+shao2010}, Shan et al.  show  that the fourth largest energy tree is either $P_{n}(2, 6, n-9)$ or $T_{n}(2, 2|2, 2)$ (see Fig.\ref{figstarlike} and Fig.\ref{figtree} for the definitions of these two graphs).
B. Huo  et al.\cite{huo2011complete} proved that the conjecture on the fourth maximal energy tree is true by showing that $\mathbb{E}(P_{n}(2, 6, n-9))> \mathbb{E}(T_{n}(2, 2|2, 2))$. Now by using Theorem \ref{thm3.1}, we are able to
give a simplified proof of the conjecture on the fourth maximal energy tree.
\begin{thm}\label{thmhuo} If $n \geq 10$, then
$$\mathbb{E}(P_{n}(2, 6, n-9))> \mathbb{E}(T_{n}(2, 2|2, 2))$$
\end{thm}
\begin{proof}
Let $H=P_{10}(2, 6, 1)$ and $G=T_{10}(2, 2|2, 2)$, $e$ be the pendent edge on the pendent path of length 1 in $H$, and $e'$ be the edge between the two vertices of degree 3 in $G$. Then we have $P_{n}(2, 6, n-9)=H(n-10)$ and $T_{n}(2, 2| 2, 2)=G(n-10)$.
By some directly calculations,  we have

$$\begin{aligned}
\widetilde{\phi}(H(0), x)&=\widetilde{\phi}(P_{10}(2, 6, 1), x)={x}^{10}+9\, {x}^{8}+27\, {x}^{6}+31\, {x}^{4}+12\, {x}^{2}+1,\\
\widetilde{\phi}(G(0), x)&=\widetilde{\phi}(T_{10}(2, 2|2, 2), x)={x}^{10}+9\, {x}^{8}+26\, {x}^{6}+30\, {x}^{4}+13\, {x}^{2}+1,\\
\widetilde{\phi}(H(1), x)&=\widetilde{\phi}(P_{11}(2, 6, 2), x)={x}^{11}+10\, {x}^{9}+35\, {x}^{7}+52\, {x}^{5}+32\, {x}^{3}+6\, x,\\
\widetilde{\phi}(G(1), x)&=\widetilde{\phi}(T_{11}(2, 2|2, 2), x)={x}^{11}+10\, {x}^{9}+34\, {x}^{7}+48\, {x}^{5}+29\, {x}^{3}+6\, x.
\end{aligned}
$$

\noindent So we have

$\widetilde{\phi}(H(1), x)\widetilde{\phi}(G(0), x)-\widetilde{\phi}(H(0), x)\widetilde{\phi}(G(1), x)=2 {x}^{15}+22 {x}^{13}+89 {x}^{11}+168 {x}^{9}+156 {x}^{7}+66 {x}^{5}+9 {x}^{3} > 0\ (x>0). $

\noindent Also by using computer we can obtain

$\mathbb{E}(H(0))\doteq 11.937511$, \quad $\mathbb{E}(G(0))\doteq11.924777$, \quad So $\mathbb{E}(H(0))-\mathbb{E}(G(0))\doteq 0.012734 >0.$

\noindent So by Theorem \ref{thm3.1} we have for $n \geq 10$,\\[0.2cm]
\hspace*{1cm}$\mathbb{E}(P_{n}(2, 6, n-9)) - \mathbb{E}(T_{n}(2, 2|2, 2))=\mathbb{E}(H(n-10))-\mathbb{E}(G(n-10))\geq \mathbb{E}(H(0))-\mathbb{E}(G(0)) >0.$
\end{proof}

Combining Theorem \ref{thmhuo} with the result that the fourth largest energy tree is either $P_{n}(2, 6, n-9)$ or $T_{n}(2, 2|2, 2)$ (\cite{shan+shao2010}), we conclude that the fourth maximal energy tree is $P_{n}(2,6,n-9)$.

\vskip 0.3cm

\noindent {\bf Remark:}
Here we would like to mention that,  the main points of the simplification in the proof of  Theorem \ref{thmhuo} are:

1. We use the integral formula (\ref{equ1.5}) (instead of (\ref{equ1.3})) which uses the real polynomial $\widetilde{\phi}(G_{j}, x)$ instead of the complex polynomial $\phi(G_{j}, ix)$ for $j=1,2$.

2. The recurrence relation $(\ref{equ2.1})$ for $\widetilde{\phi}(G(k), x)$ allows us to use Lemma \ref{lem3.1} to directly compare $d_{k}(x)$ and $d_{0}(x)$ (namely directly compare the integrands $\ln d_{k}(x)$ and $\ln d_{0}(x)$ in the formula (\ref{equ1.5})
for $\mathbb{E}(H(k))- \mathbb{E}(G(k))$ and $\mathbb{E}(H(0))- \mathbb{E}(G(0))$),  without the need of solving  the recurrence relation $(\ref{equ2.1})$  to obtain explicit expressions for $h_{k}=\widetilde{\phi}(H(k), x)$ and $g_{k}=\widetilde{\phi}(G(k), x)$. \qed

\vskip 0.3cm

Notice that in Theorem \ref{thm3.1}, we need either $d_{1}(x) > d_{0}(x)$ for all $x >0$ or $d_{0}(x) > d_{1}(x)$ for all $x >0$. Now if both of these two conditions do not hold,  then both $d_{0}(x)$ and $d_{1}(x)$ are not a lower bound for $d_{k}(x)$ ($k \geq 2$). Although in this case we can not use Theorem \ref{thm3.1}, but by Lemma \ref{lem3.1} we still have $\min\{d_{0}(x),d_{1}(x)\}$ as a lower bound for $d_{k}(x)$  (for all $x >0$). Thus we can still have the following lower bound (which is independent of $k$) for  $\mathbb{E}(H(k))- \mathbb{E}(G(k))$.

\begin{thm}\label{engergycompare}
Let $G(k)$, $H(k)$ be $k$-subdivision graphs of   bipartite graphs $G$ and $H$ on some cut edges. Let $d_{k}(x)= \displaystyle\frac{\widetilde{\phi}(H(k), x)}{\widetilde{\phi}(G(k), x)}$ and let
 ${D}=\{x>0 |d_{0}(x) > d_{1}(x)\}$,  Let ${D^{C}}$ be the complement of $D$ in $(0, \infty)$.  Then :
 \begin{equation}\label{equ3.1}
 \mathbb{E}(H(k))- \mathbb{E}(G(k))\geq
 \frac{2}{\pi}\int\limits_{0}^{+\infty}\ln  \min\{d_{0}(x),d_{1}(x)\}\text{d}x
  =\frac{2}{\pi}\int\limits_{D}\ln  d_{1}(x)\text{d}x+\frac{2}{\pi}\int\limits_{D^{C}}\ln d_{0}(x) \text{d}x
 \end{equation}
where the right hand side of (\ref{equ3.1}) can also be written as:

 \begin{align}\label{equ3.2}
 &\frac{2}{\pi}\int\limits_{D}\ln  d_{1}(x)\text{d}x+\frac{2}{\pi}\int\limits_{D^{C}}\ln d_{0}(x) \text{d}x
  = \frac{2}{\pi}\int\limits_{0}^{+\infty}\ln d_{1}(x)\text{d}x - \frac{2}{\pi}\int\limits_{D^{C}}\ln d_{1}(x)\text{d}x  + \frac{2}{\pi}\int\limits_{D_{C}}\ln d_{0}(x)\text{d}x\notag\\
 =&
 \mathbb{E}(H(1))-\mathbb{E}(G(1))-\frac{2}{\pi}\int\limits_{D^{C}}\ln\frac{d_{1}(x)}{d_{0}(x)}\text{d}x
 \end{align}
 \noindent or equivalently,
 \begin{align}\label{equ3.3}
 &\frac{2}{\pi}\int\limits_{D}\ln  d_{1}(x)\text{d}x+\frac{2}{\pi}\int\limits_{D^{C}}\ln d_{0}(x) \text{d}x
  =\mathbb{E}(H(0))-\mathbb{E}(G(0))+\frac{2}{\pi}\int\limits_{{D}}\ln\frac{d_{1}(x)}{d_{0}(x)}\text{d}x
 \end{align}
\end{thm}

\noindent Theorem \ref{engergycompare} will be used several times in \S 4 and \S 5 in the proof of our main results.

\section{Some upper bounds for the energies of non-starlike trees }

In the following discussions, we will divide the trees into two classes. One is called the starlike trees, and the other one is the non-starlike trees. In this section,  We will give some upper bounds for  the energies of the non-starlike trees. We will show that the energy of a non-starlike tree is bounded above either by the energy of $P_n(1,2,n-4)$, or by the energy of $T_n(2,2|2,2)$ (see Fig.\ref{figstarlike} and Fig.\ref{figtree}).

 Let $N_{3}(G)$ be the number of vertices in $G$ with degree at least 3, and $\Delta (G)$ be the maximal degree of $G$. A tree $T$ is called  starlike if $N_{3}(T)\le 1$, and is called non-starlike if $N_{3}(T)\ge 2$.

 It is easy to see that if $N_{3}(T)=0$, then $T$ is the path $P_n$. Now if $N_{3}(T)=1$, then $T$ consists of some internally disjoint pendent paths starting from its unique vertex with degree at least 3. Suppose that the lengths of these pendent paths are positive integers $a_{1},  a_{2}, \cdots,  a_{k}$.
Then we denote this tree $T$ by $P_{n}(a_{1},  a_{2}, \cdots,  a_{k})$,  where $a_{1}+a_2+\cdots+a_{k}=n-1$ and $k=\Delta (T)$ (see Fig.\ref{figstarlike}). Sometimes we also denote
$P_{n}(a_{1},  a_{2}, \cdots,  a_{k})$ by   $P_{n}(a_{1},  a_{2}, \cdots,  a_{k-1}, *)$,  since  $*$  is uniquely
determined by $n$ and $a_{1},  a_{2}, \cdots,  a_{k-1}$.

\begin{figure}[h]
 \begin{minipage}[t]{0.5\linewidth}
 \centering
\begin{center}
\begin{tikzpicture}

 \coordinate (O) at (0:0);
 \foreach \r/\n in {0.6/1, 1.4/2, 2.0/3}
 \coordinate (A\n) at (-45:\r);

 \foreach \r/\n in {0.6/1, 1.4/2, 2.0/3}
 \coordinate (B\n) at (-15:\r);

 \foreach \r/\n in {0.6/1, 1.4/2, 2.0/3}
 \coordinate (C\n) at (15:\r);

 \foreach \r/\n in {0.6/1, 1.4/2, 2.0/3}
 \coordinate (D\n) at (45:\r);

 \foreach \point in {A, B, C, D}{
 \foreach \n in {1, 2, 3}
 {\fill [black] (\point\n) circle (1.5pt);}
 }
\fill [black] (O) circle (1.5pt);

 \foreach \n in {A, B, C, D}{
\draw (O)--(\n1)(\n2)--(\n3);
\draw [decoration={brace, raise=2.5pt},  decorate] (\n1) -- (\n3);
}
\foreach \d in {-45, -15, 15, 45}{
\node[rotate=\d] at (\d:1) {$\cdots$};
}

\node[rotate=90] at (0:2) {$\cdots$};

\node  at (65:1.3) {\small$P_{a_{1}}$};

\node  at (32:1.3) {\small$P_{a_{2}}$};

\node  at (0:1.4) {\small$P_{a_{k-1}}$};

\node  at (-30:1.4) {\small$P_{a_{k}}$};

\end{tikzpicture}
\end{center}
  \vspace*{-0.6cm}\caption{The starlike tree $P_{n}(a_{1},  a_{2}, \cdots,  a_{k})$}\label{figstarlike}
 \end{minipage}%
 \begin{minipage}[t]{0.5\linewidth}
\begin{center}
\begin{tikzpicture}
\tikzstyle myline=[line width=0.8pt]
\coordinate (A) at (-1.2, 0);
\coordinate (B) at (-0.5, 0);
\coordinate (C) at (0.5, 0);
\coordinate (D) at (1.2, 0);
\coordinate (E) at (-1.5, 0.3);
\coordinate (F) at (-2.1, 0.9);
\coordinate (G) at (-2.4, 1.2);
\coordinate (E1) at (1.5, 0.3);
\coordinate (F1) at (2.1, 0.9);
\coordinate (G1) at (2.4, 1.2);
\coordinate (E2) at (-1.5, -0.3);
\coordinate (F2) at (-2.1, -0.9);
\coordinate (G2) at (-2.4, -1.2);
\coordinate (E3) at (1.5, -0.3);
\coordinate (F3) at (2.1, -0.9);
\coordinate (G3) at (2.4, -1.2);
\node[rotate=135] at (-1.8, 0.6) {$\cdots$};
\node[rotate=45] at (1.8, 0.6) {$\cdots$};
\node[rotate=-45] at (1.8, -0.6) {$\cdots$};
\node[rotate=-135] at (-1.8, -0.6) {$\cdots$};
\draw (A)--(B);
\draw (C)--(D);
\node at (0, 0) {$\cdots$};
\draw (F)--(G);
\draw (F3)--(G3);
\draw (F2)--(G2);
\draw (F1)--(G1);
\draw (A)--(E2);
\draw (A)--(E);
\draw (D)--(E1);
\draw (D)--(E3);
\draw [decoration={brace, raise=2.5pt},  decorate] (G) -- (E);
\draw [decoration={brace, raise=2.5pt},  decorate] (G2) -- (E2);
\draw [decoration={brace, raise=2.5pt},  decorate] (E1) -- (G1);
\draw [decoration={brace, raise=2.5pt},  decorate] (E3) -- (G3);
\node at (-1.6, 1.0) {$P_{a}$};
\node at (-2.3, -0.5) {$P_{ b}$};
\node at (1.6, 1.0) {$P_{c}$};
\node at (2.3, -0.5) {$P_{d}$};
\foreach \point in {A, B, C, D, E, F, G, E1, F1, G1, E2, F2, G2, E3, F3, G3}
{\fill [black] (\point) circle (1.5pt);}
\end{tikzpicture}
\end{center}
    \vspace*{-0.9cm}\caption{The tree $T_{n}(a, b|c, d)$}\label{figtree}
 \end{minipage}
 \end{figure}

Let $a, b, c, d$ be positive integers with $a+b+c+d\le n-2$. Let $T_{n}(a, b|c, d)$ be the tree of order $n$ obtained by attaching two pendent paths of lengths $a$ and $b$ to one end vertex of the path $P_{n-a-b-c-d}$,
and attaching two pendent paths of lengths $c$ and $d$ to another end vertex of the path $P_{n-a-b-c-d}$ (see Fig.\ref{figtree}).

It is not difficult to see that if $T$ is a tree of order $n$ with $\Delta(T)=3$ and $N_{3}(T)=2$,  then $T$ must be of the form $T_{n}(a, b|c, d)$, where $a+b+c+d\le n-2$.

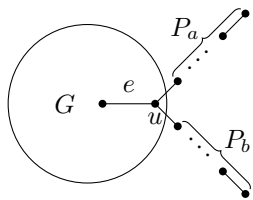
\begin{figure}[h]
\begin{center}
\begin{tikzpicture}
\tikzstyle myline=[line width=0.8pt]
\coordinate (C) at (0.5,0);
\coordinate (D) at (1.2,0) ;
\coordinate (E1) at (1.5,0.3);
\coordinate (F1) at (2.1,0.9);
\coordinate (G1) at (2.4,1.2);
\coordinate (E2) at (-1.5,-0.3);
\coordinate (F2) at (-2.1,-0.9);
\coordinate (G2) at (-2.4,-1.2);
\coordinate (E3) at (1.5,-0.3);
\coordinate (F3) at (2.1,-0.9);
\coordinate (G3) at (2.4,-1.2);
\node at (1.2,-0.2){$u$};
\node[rotate=45] at (1.8,0.6) {$\cdots$};
\node[rotate=-45] at (1.8,-0.6) {$\cdots$};
\draw (C)--(D)  node [midway, above=1pt] {$e$};;
\draw (F1)--(G1);
\draw (F3)--(G3);
\draw (D)--(E1);
\draw (D)--(E3);
\draw [decoration={brace,raise=2.5pt}, decorate] (E1) -- (G1);
\draw [decoration={brace,raise=2.5pt}, decorate] (E3) -- (G3);
\node at (1.6,1.0) {$P_{a}$};
\node at (2.3,-0.5) {$P_{b}$};
\node at (0,0) {$G$};
\foreach \point in {C,D,E1,F1,G1,E3,F3,G3}
{\fill [black] (\point) circle (1.5pt);}
\draw (0.3,0) circle (30pt);
\end{tikzpicture}
\end{center}\vspace{-0.5cm}
  \caption{The graph $G_{u}(a,b)$}\label{figsamegraft}
\end{figure}

In \cite{shan+shao2010} and \cite{shangraftUOB},  Shan et al. studied how graph energies change under edge grafting operations on unicyclic or bipartite graphs and proved the following result in the comparison of the quasi-order on unicyclic or bipartite graphs:

\begin{lem}(\cite{shan+shao2010}, The edge grafting operation)\label{samegraft}
Let $u$ be a vertex of a graph $G$.  Denote $G_{u}(a, b)$  the graph  obtained by attaching to $G$ two (new) pendent paths
of lengths $a$ and $b$ at $u$. Let $a, b, c, d$ be nonnegative integers with $a+b=c+d$. Assume that  $0 \le a \le b,  \  0 \le c \le d$ and $a<c$.
If $u$ is a non-isolated vertex of a unicyclic or bipartite graph $G$, then the following statements are true:
\begin{enumerate}[(1).]
    \item If $a$ is even,  then $G_{u}(a, b) \succ G_{u}(c, d)$.
    \item If $a$ is odd,  then  $G_{u}(a, b) \prec G_{u}(c, d)$.
\end{enumerate}
\end{lem}

If $a=0$,  then we say that $G_{u}(0, b)$ is obtained from $G_{u}(c, d)$ by a $total$ $edge$ $grafting$ operation.

\vskip 0.2cm

The following result in \cite{shan+shao2010} was obtained directly by using the edge grafting operation.

\begin{thm}\cite{shan+shao2010}\label{thmrich}
 Let $T$  be a tree of order $n$ with $N_3(T)\ge 2$.  Then there exists a tree $T'$ of order $n$ with $N_{3}(T')=N_{3}(T)-1$ and $\Delta(T')=\Delta(T)$ such that\ $T \prec T'$.
\end{thm}

In the followings, we will give some upper bounds for the energies of the trees of the form $T_{n}(a, b|c, d)$. First we consider the case $1\in \{a,b,c,d\}$ in the following Theorem \ref{thm4.2}. The other case where $min \{a,b,c,d\}\ge 2$ will be considered in Lemma \ref{lem4.3}, \ref{lem4.4} and Theorem \ref{thmtabcd}.

\begin{thm}\label{thm4.2}\cite{shan+shao2010}
 Let $T=T_{n}(1, b|c, d)$. Then $T\prec P_n(1,2,n-4)$.
\end{thm}

\begin{proof} By using total edge grafting on the two pendent paths of lengths $c$ and $d$, we have $T\prec P_n(1,b,n-2-b)$. Using the edge grafting operation again, we have $P_n(1,b,n-2-b)\preceq P_n(1,2,n-4)$. Thus the result follows.
\end{proof}

\vskip 0.2cm

\begin{figure}[h]
\vspace{-0.5cm}
\begin{center}
\begin{tikzpicture}[scale=0.8, yshift=2cm]
\tikzstyle myline=[line width=0.8pt]
\coordinate (A1) at (0, 0.2);
\coordinate (B1) at (0.6, 0.5);
\coordinate (C1) at (1.2, 0.8);
\coordinate (D1) at (1.8, 1.1);
\node[rotate=27] at (0.95, 0.65) {$\cdots$};
\draw (A1)--(B1);
\draw (C1)--(D1);
\draw [decoration={brace, raise=2.5pt},  decorate] (B1) -- (D1)node [midway,  above=7pt]{$P_{a}$};
\coordinate (A2) at (0, -0.2);
\coordinate (B2) at (0.6, -0.5);
\coordinate (C2) at (1.2, -0.8);
\coordinate (D2) at (1.8, -1.1);
\node[rotate=-27] at (0.95, -0.65) {$\cdots$};
\draw (A2)--(B2);
\draw (C2)--(D2);
\draw [decoration={brace, raise=2.5pt},  decorate] (B2) -- (D2)node [midway,  above=7pt]{$P_{b}$};
\foreach \point in {A1, B1, C1, D1, A2, B2, C2, D2}
{\fill [black] (\point) circle (1.5pt);}
\draw (-0.7, 0) circle (30pt);

\node[left] at (A1) {$u$};
\node[left] at (A2) {$v$};
\node at (-1, 0) {$G$};
\end{tikzpicture}
\end{center}
    \vspace*{-0.5cm}\caption{$G_{u, v}(a, b)$}\label{figdif}
\end{figure}
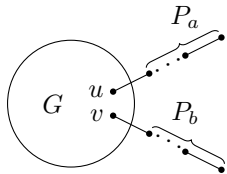

The following Lemma generalizes Lemma \ref{samegraft}, and is called ``edge grafting operation at different vertices''.

\begin{lem}\cite{shangraftUOB}\label{difgrift}
Let $u, v$ be two vertices of a unicyclic or bipartite graph $G$. Let $G_{u, v}(a, b)$ be the graph obtained from $G$ by attaching a pendent path of length $a$ to $u$ and attaching a pendent path of length $b$ to $v$ (as shown in Fig.\ref{figdif}). Suppose that $G$ satisfies:

(i). $G_{u, v}(0, 2)\succ G_{u, v}(1, 1)$.

(ii). For any nonnegative integers
{$p, q$},
$G_{u, v}(p, q)=G_{u, v}(q, p)$.

\noindent Let $a, b, c, d$  be nonnegative integers with $a\le b,  \
c\le d,  \ a+b=c+d$,  and $a <c$,  then we have

(1) If $a$ is even,  then $G_{u, v}(a, b)\succ G_{u, v}(c, d)$.

(2) If $a$ is odd,  then $G_{u, v}(a, b)\prec G_{u, v}(c, d)$.
\end{lem}

\begin{figure}[h]
\begin{center}
\begin{tikzpicture}[scale=0.9]
\tikzstyle myline=[line width=0.8pt]
\begin{scope}[xshift=-6cm]
\coordinate (A) at (-1.2, 0);
\coordinate (B) at (0, 0);
\coordinate (C) at (0, 0);
\coordinate (D) at (1.2, 0);

\coordinate (E2) at (1.6, 0.4);
\coordinate (F2) at (2, 0.8);
\coordinate (E3) at (1.6, -0.4);
\coordinate (F3) at (2, -0.8);
\coordinate (E4) at (-1.6, -0.4);
\coordinate (F4) at (-2, -0.8);

\coordinate (E1) at (-1.6, 0.4);
\coordinate (F1) at (-2, 0.8);
\coordinate (G) at (-2.4, 1.2);

\draw (G)--(F1);

\draw (A)--(B);
\draw (C)--(D);
\draw (A)--(E1);
\draw (A)--(E4);
\draw (D)--(E2);
\draw (D)--(E3);

\draw (E2)--(F2);
\draw (E3)--(F3);
\draw (E4)--(F4);
\node[above] at (0.8, 0) {$e_{2}$};
\node[above] at (-1.22, 0.1) {$e_{1}$};
\draw (G) -- (E1);

\foreach \point in {A, B, C, D, E1, F1, G, E2, F2, E3, F3, E4, F4}
{\fill [black] (\point) circle (1.5pt);}
\end{scope}

\begin{scope}[xshift=2cm]
\coordinate (A) at (-1.2, 0);
\coordinate (B) at (-0.4, 0);
\coordinate (C) at (0.4, 0);
\coordinate (D) at (1.2, 0);
\coordinate (E1) at (-1.6, 0.4);
\coordinate (F1) at (-2, 0.8);
\coordinate (E2) at (1.6, 0.4);
\coordinate (F2) at (2, 0.8);
\coordinate (E3) at (1.6, -0.4);
\coordinate (F3) at (2, -0.8);
\coordinate (E4) at (-1.6, -0.4);
\coordinate (F4) at (-2, -0.8);
\draw (A)--(B)--(C)--(D);
\draw (A)--(E1);
\draw (A)--(E4);
\draw (D)--(E2);
\draw (D)--(E3);
\draw (E1)--(F1);
\draw (E2)--(F2);
\draw (E3)--(F3);
\draw (E4)--(F4);

\node [above] at (-0.8, 0.08) {$e_{1}'$};
\node [above] at (0, 0.08) {$e_{2}'$};

\foreach \point in {A, B, C, D, E1, F1, E2, F2, E3, F3, E4, F4}
{\fill [black] (\point) circle (1.5pt);}
\end{scope}
\end{tikzpicture}
\end{center}
  \vspace*{-0.5cm}\caption{ $T_{12}(3, 2|2, 2)$ and $T_{12}(2, 2|2, 2)$}\label{figsubcase3.2}
\end{figure}
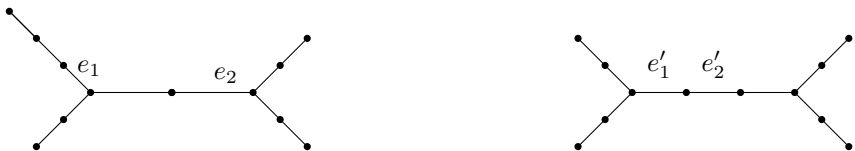

Now we use the methods given in \S 3 to prove the following two lemmas, which consider the tree $T_{n}(a, 2|2, 2)$ in two cases $3 \leq  a  \leq n-9$ and $a=n-8$. These two lemmas will only be used in the proof of the  Theorem \ref{thmtabcd} later.

\begin{lem}\label{lem4.3}
Let $3 \leq  a  \leq n-9$. Then   $T_{n}(a, 2|2, 2) \prec T_{n}(2, 2|2, 2)$.
\end{lem}

\begin{proof}
Let $e_{1},  e_{2}$ be the cut edges of  $G=T_{12}(3, 2|2, 2)$  and $e'_{1}$, $e'_{2}$ be the cut edges of  $H=T_{12}(2, 2|2, 2)$   as shown in Fig.\ref{figsubcase3.2}. respectively.
 Then we have $T_{n}(a, 2|2, 2)= G(a-3,n-9-a) $ and $ T_{n}(2, 2|2, 2)= H(a-3,n-9-a)$.

By some directly calculations,  we have
$$\begin{aligned}
\widetilde{\phi}(H(0,0), x)=&\widetilde{\phi}(T_{12}(2, 2|2, 2), x)={x}^{12}+11\, {x}^{10}+43\, {x}^{8}+74\, {x}^{6}+59\, {x}^{4}+19\, {x}^{2}+1, \\
\widetilde{\phi}(G(0,0), x)=&\widetilde{\phi}(T_{12}(3, 2|2, 2), x)={x}^{12}+11\, {x}^{10}+43\, {x}^{8}+74\, {x}^{6}+57\, {x}^{4}+17\, {x}^{2}, \\
\widetilde{\phi}(H(1,0), x)=\widetilde{\phi}(H(0,1), x)=&\widetilde{\phi}(T_{13}(2, 2|2, 2), x)={x}^{13}+12\, {x}^{11}+53\, {x}^{9}+108\, {x}^{7}+107\, {x}^{5}+48\, {x}^{3}+7\, x, \\
\widetilde{\phi}(G(1,0), x)=&\widetilde{\phi}(T_{13}(4, 2|2, 2), x)={x}^{13}+12\, {x}^{11}+53\, {x}^{9}+108\, {x}^{7}+105\, {x}^{5}+46\, {x}^{3}+7\, x, \\
\widetilde{\phi}(G(0,1), x)=&\widetilde{\phi}(T_{13}(3, 2|2, 2), x)={x}^{13}+12\, {x}^{11}+53\, {x}^{9}+108\, {x}^{7}+106\, {x}^{5}+46\, {x}^{3}+6\, x, \\
\widetilde{\phi}(H(1,1), x)=&\widetilde{\phi}(T_{14}(2, 2|2, 2), x)={x}^{14}+13\, {x}^{12}+64\, {x}^{10}+151\, {x}^{8}+181\, {x}^{6}+107\, {x}^{4}+26\, {x}^{2}+1, \\
\widetilde{\phi}(G(1,1), x)=&\widetilde{\phi}(T_{14}(4, 2|2, 2), x)={x}^{14}+13\, {x}^{12}+64\, {x}^{10}+151\, {x}^{8}+180\, {x}^{6}+105\, {x}^{4}+25\, {x}^{2}+1.\\
\end{aligned}$$
By comparing the coefficients of above polynomials,  we find that
$$G(0, 0)\prec  H(0, 0), \ G(0, 1)\prec  H(0, 1), \  G(1, 0)\prec  H(1, 0), \ G(1, 1)\prec  H(1, 1).$$
So by Theorem \ref{thm2.4} we have $T_{n}(a, 2|2, 2)=G(a-3,n-9-a)\prec H(a-3,n-9-a)=T_{n}(2, 2|2, 2)$.
\end{proof}

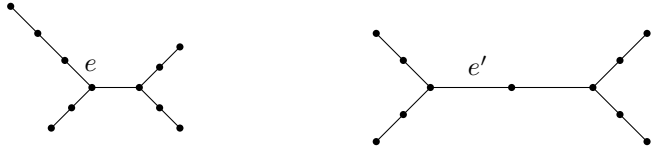
\begin{figure}[h]
\begin{center}
\begin{tikzpicture}[scale=0.9]
\tikzstyle myline=[line width=0.8pt]
\begin{scope}[xshift=1cm]
\tikzstyle myline=[line width=0.8pt]
\coordinate (A) at (-1.2, 0);
\coordinate (B) at (-0.5, 0);
\coordinate (E) at (-1.6, 0.4);
\coordinate (F) at (-1.5, -0.3);
\coordinate (G) at (-2.4, 1.2);
\coordinate (F3) at (-2, 0.8);

\coordinate (E1) at (-0.2, 0.3);
\coordinate (F1) at (-0.2, -0.3);
\coordinate (E2) at (0.1, 0.6);
\coordinate (F2) at (0.1, -0.6);

\coordinate (E3) at (-1.8, -0.6);
\draw (E3)--(F)--(A)--(B)--(E1)--(E2);
\draw (B)--(F1)--(F2);
\draw (A)--(E);
\draw (E)--(F3)--(G);
\node[above] at (-1.22, 0.1) {$e$};

\foreach \point in {A, B, E, F, G, E1, F1, E2, F2, E3, F3}
{\fill [black] (\point) circle (1.5pt);}
\end{scope}

\begin{scope}[xshift=6cm]
\tikzstyle myline=[line width=0.8pt]
\coordinate (A) at (-1.2, 0);
\coordinate (B) at (0, 0);
\coordinate (D) at (1.2, 0);
\coordinate (E1) at (-1.6, 0.4);
\coordinate (F1) at (-2, 0.8);
\coordinate (E2) at (1.6, 0.4);
\coordinate (F2) at (2, 0.8);
\coordinate (E3) at (1.6, -0.4);
\coordinate (F3) at (2, -0.8);
\coordinate (E4) at (-1.6, -0.4);
\coordinate (F4) at (-2, -0.8);
\draw (A)--(B)--(D);
\draw (A)--(E1);
\draw (A)--(E4);
\draw (D)--(E2);
\draw (D)--(E3);
\draw (E1)--(F1);
\draw (E2)--(F2);
\draw (E3)--(F3);
\draw (E4)--(F4);
\node [above] at (-0.5, 0.04) {$e'$};
\foreach \point in {A, B, D, E1, F1, E2, F2, E3, F3, E4, F4}
{\fill [black] (\point) circle (1.5pt);}
\end{scope}

\end{tikzpicture}
\end{center}
\vspace{-0.5cm}
  \caption{$G=T_{11}(3, 2|2, 2)$ and $H=T_{11}(2, 2|2, 2)$}\label{figsubcase3.1}
\end{figure}

Now we consider the remaining case $a=n-8$ for the trees of the form $T_{n}(a, 2|2, 2)$.

\begin{lem}\label{lem4.4}
$\mathbb{E}(T_{n}(n-8, 2|2, 2)) < \mathbb{E}({ T_{n}(2, 2|2, 2)})$ for all $n\ge 11$.
\end{lem}

\begin{proof}
Consider the cut edges $e$ of $G=T_{11}(3, 2|2, 2)$ and $e'$ of $H=T_{11}(2, 2|2, 2)$ as shown in Fig.\ref{figsubcase3.1}.
Let $G(k)$, $H(k)$ be graphs obtained by subdividing  the cut edges $e$ of $G$ and $e'$ of $H$ respectively $k$ times.
 Then we have $T_{n}(n-8, 2|2,2)=G(n-11)$ and $T_{n}(2, 2, 2, 2)=H(n-11)$.
 Denote $g_{k}= \widetilde{\phi}(G(k), x)$ and $h_{k}= \widetilde{\phi}(H(k), x)$.

By some directly calculations,  we have
$$\begin{aligned}
h_{0}=&\widetilde{\phi}(T_{11}(2, 2|2, 2), x)={x}^{11}+10\, {x}^{9}+34\, {x}^{7}+48\, {x}^{5}+29\, {x}^{3}+6\, x,\\
g_{0}=&\widetilde{\phi}(T_{11}(3, 2|2, 2), x)={x}^{11}+10\, {x}^{9}+34\, {x}^{7}+49\, {x}^{5}+29\, {x}^{3}+5\, x,\\
h_{1}=&\widetilde{\phi}(T_{12}(2, 2|2, 2), x)={x}^{12}+11\, {x}^{10}+43\, {x}^{8}+74\, {x}^{6}+59\, {x}^{4}+19\, {x}^{2}+1,\\
g_{1}=&\widetilde{\phi}(T_{12}(4, 2|2, 2), x)={x}^{12}+11\, {x}^{10}+43\, {x}^{8}+75\, {x}^{6}+59\, {x}^{4}+18\, {x}^{2}+1.\\
\end{aligned}$$

\noindent So we have
$$ h_{1}g_{0}-h_{0}g_{1}=x(x-1)( x+1) ({x}^{6}+7\, {x}^{4}+11\, {x}^{2}+1)( {x}^{2}+1) ^{3}.$$
Thus
$${D}=\{x|h_{1}g_{0}- h_{0}g_{1} < 0,  x>0\} = (0,  1).$$
Also by using computer we can find:
$$\mathbb{E}(H(0))\doteq 13.059967,  \quad\quad \mathbb{E}(G(0))\doteq 13.015698$$
and by using computer to calculate the integral we can further obtain $$\displaystyle\mathbb{E}(H(0))-\mathbb{E}(G(0))+\frac{2}{\pi}\int\limits_{{D}}\ln\frac{d_{1}(x)}{d_{0}(x)}\text{d}x =\mathbb{E}(H)-\mathbb{E}(G)+\frac{2}{\pi}\int\limits_{{0}}^1\ln\frac{h_{1}g_{0}}{h_{0}g_{1}}\text{d}x \doteq 0.005951>0.$$
So using  Theorem \ref{engergycompare},  we obtain $\mathbb{E}(H(k))-\mathbb{E}(G(k))> 0$ for all $k \ge 0 $.
Thus $\mathbb{E}(T_{n}(n-8, 2|2, 2)) < \mathbb{E}({ T_{n}(2, 2|2, 2)})$.
\end{proof}

\begin{thm}\label{thmtabcd}
Let $n \geq 11$,  $a,b,c,d\ge 2$ and $a,b,c,d$ are not all equal to 2. Then we have $$\mathbb{E}({T_{n}(a, b|c, d)}) < \mathbb{E}({ T_{n}(2, 2|2, 2)}).$$
\end{thm}

\begin{proof}

By using the edge grafting operation in Lemma \ref{samegraft}, we have
$$T_{n}(a, b|c, d) \preccurlyeq T_{n}(a+b-2, 2|2, c+d-2).$$
By using Lemma \ref{difgrift} (edge grafting on different vertices), we also have
$$T_{n}(a+b-2, 2|2, c+d-2)\preccurlyeq T_{n}(a+b+c+d-6, 2|2, 2).$$
Write $x=a+b+c+d-6$, then we have $3\le x\le n-8$ since at least one of $a,b,c,d$ is greater than 2.

Now  If $3 \leq  x  \leq n-9$, then by Lemma \ref{lem4.3} we have $T_{n}(x, 2|2, 2) \prec T_{n}(2, 2|2, 2)$. So $\mathbb{E}({T_{n}(a, b|c, d)})\leq \mathbb{E}(T_{n}(x, 2|2, 2)) < \mathbb{E}({ T_{n}(2, 2|2, 2)})$.

If $x = n-8$,  then by Lemma \ref{lem4.4} we have $\mathbb{E}({T_{n}(a, b|c, d)})\leq \mathbb{E}(T_{n}(x, 2|2, 2)) < \mathbb{E}({ T_{n}(2, 2|2, 2)})$.
\end{proof}

\section{The trees of order $n$   with the first $\lfloor\frac{n-7}{2}\rfloor$ largest energies}

In this section,  we will determine  the first $\lfloor\frac{n-7}{2}\rfloor$  largest energy trees of order $n\ge 31$ by using the method of directly comparing energies given in \S 3.

First, we  divide the class of starlike trees into the following four subclasses:

\vskip 0.2cm

\noindent {\bf (C1).} The path $P_n$.

\vskip 0.2cm

\noindent {\bf (C2).} The class  $S_{n}= \{P_{n}(2, a, b) \ | \ a+b=n-3,  \  1\leq a \leq  b \}$.

\vskip 0.2cm

\noindent {\bf (C3).} The starlike trees $T$ of order $n$ with $\Delta (T)=3$ and $T\notin S_n$.

\vskip 0.2cm

\noindent {\bf (C4).} The starlike trees $T$ of order $n$ with $\Delta (T)\ge 4$.

\vskip 0.2cm

\noindent For convenience, we also define the following class (C5):

\vskip 0.2cm

\noindent {\bf (C5).} The class of non-starlike trees of order $n$ (i.e., $N_3(T)\ge 2$).

\vskip 0.2cm

\noindent It is obvious that the union of the classes (C1)-(C5) is the class of all the trees of order $n$.

\vskip 0.2cm

Now, our strategy of proving the main result is as follows. Firstly, using the quasi-order we can obtain (in Theorem \ref{thm5.1}) a total ordering of all the $\lfloor\frac{n-3}{2}\rfloor$ trees in $S_n$. Secondly, we can show (in Theorem \ref{thm5.2}) that the maximal tree (under the quasi-order) in the class (C3) is $P_n(4,4,*)$, and the maximal tree  in the class (C4) is $P_n(2,2,2,*)$. Next, by directly comparing the energies of the largest energy trees in the classes (C3) and (C4) with some smaller energy  graphs in $S_n$, and comparing the energies of the  tree $T_n(2,2|2,2)$ in the class (C5) with the smallest energy tree $P_n(2, 1,n-4)$ in $S_n$,  we obtain that the first $\lfloor\frac{n-9}{2}\rfloor$ largest energy trees in $S_n$ together with $P_n$ are the first $\lfloor\frac{n-7}{2}\rfloor$ largest energy trees in the class of all trees of order $n$.

\begin{thm}\label{thm5.1}
Let $S_{n}= \{P_{n}(2, a, b) \ | \ a+b=n-3,  \  1\leq a \leq  b\}$. Let $k=\lfloor\frac{n-3}{2}\rfloor$,  $t=\lfloor\frac{k}{2}\rfloor$ and $l=\lfloor\frac{k-1}{2}\rfloor$. Then we have the following totally quasi order for the trees in $S_n$:
\begin{equation}\label{equ5.1}
P_{n}(2, 2, *) \succ P_{n}(2, 4, *)  \succ \cdots \succ P_{n}(2, 2t, *)\succ P_{n}(2, 2l+1, *)\succ \cdots  \succ  P_{n}(2, 3, *) \succ  P_{n}(2, 1, *).
\end{equation}
\end{thm}
\begin{proof}
The result follows directly from Lemma \ref{samegraft} by using the edge grafting operation.
\end{proof}

\begin{thm}\label{thm5.2} Let $n\ge 11$. Then we have

\noindent (1). If $T\in $(C3) and $T\ne P_n(4,4,n-9)$, then $T\prec P_n(4,4,n-9)$ .

\noindent (2). If $T\in $(C4) and $T\ne P_n(2,2,2,n-7)$, then $T\prec P_n(2,2,2,n-7)$ .

\end{thm}

\begin{proof} (1) Since $T\in $(C3), $T$ must be of the form $P_{n}(a, b, c)$ with $2\notin \{a,b,c\}$.
Without loss of generality,  we may assume that $a \leq b \leq c$. Then $b+c \geq 7$ since $n \ge 11$. So by Lemma \ref{samegraft} we have
$T=P_{n}(a, b, c) \preccurlyeq P_{n}(a, 4, b+c-4 )$  and $P_{n}(a, 4, b+c-4 ) \preccurlyeq P_{n}(4, 4, n-9)$ since $b+c-4 \neq 2$. Also $T\ne P_n(4,4,n-9)$ implies at least one of the above two relations is strict. Thus we have  $T=P_{n}(a, b, c) \prec P_{n}(4, 4, n-9)$.

(2) Since $\Delta(T) \ge 4$ for $T\in $(C4), by using  Lemma \ref{samegraft}
we can derive that  $T \preccurlyeq P_n(a, b, c, d)$ for some tree $P_n(a, b, c, d)$.  By further using the edge grafting operations at most 3 times on $P_n(a, b, c, d)$,  we will finally obtain
$P_n(a, b, c, d)\preccurlyeq P_{n}(2, 2, 2, n-7)$.
Also $T\ne P_n(2,2,2,n-7)$ implies at least one of the above relations is strict. Thus we have $T \prec P_{n}(2, 2, 2, n-7)$.
\end{proof}

The following Theorem \ref{thm5.3} and Theorem \ref{thm5.4} will exclude out $P_{n}(2,2,2,*)$ (the maximal energy tree in the class (C4)) and $T_{n}(2,2|2,2)$ (in  (C5)) by the smallest energy tree in $S_{n}$ by using the method of  directly comparing energies given in \S 3.

\begin{figure}[h]
\begin{center}
\begin{tikzpicture}[xshift=-3cm]
\tikzstyle myline=[line width=0.8pt]
\begin{scope}[xshift=0cm]
\coordinate (A) at (-1.2, 0);
\coordinate (B) at (-0.5, 0);
\coordinate (C) at (0.4, 0);
\coordinate (E1) at (-1.6, 0.4);
\coordinate (F1) at (-2, 0.8);
\coordinate (E4) at (-1.6, -0.4);
\coordinate (F4) at (-2, -0.8);
\coordinate (E3) at (-1.7, 0);
\coordinate (F3) at (-2.2, 0);
\draw (A)--(B)--(C);
\draw (A)--(E1);
\draw (A)--(E4);
\draw (E1)--(F1);
\draw (A)--(E3);
\draw (F3)--(E3);
\draw (E4)--(F4);
\node[above] at (-0.8, 0 ) {$e$};
\foreach \point in {A, B, C, E1, F1, E4, E3, F3, F4}
{\fill [black] (\point) circle (1.5pt);}
\end{scope}

\begin{scope}[xshift=6cm]
\tikzstyle myline=[line width=0.8pt]
\coordinate (A) at (-1.2, 0);
\coordinate (B) at (-0.5, 0);
\coordinate (C) at (0.5, 0);
\coordinate (D) at (1.2, 0);
\coordinate (E1) at (-1.6, 0.4);
\coordinate (F1) at (-2, 0.8);
\coordinate (E4) at (-1.5, -0.3);
\draw (A)--(B);
\draw (C)--(D);
\draw (A)--(E1);
\draw (A)--(E4);
\draw (E1)--(F1);
\node at (0, 0) {$\cdots$};
\node[above] at (-0.8, 0 ) {$e'$};
\draw [decoration={brace, raise=2.5pt},  decorate] (B) -- (D);
\node [right] at (0, 0.5) {$P_{5}$};
\foreach \point in {A, B, C, D, E1, F1, E4}
{\fill [black] (\point) circle (1.5pt);}
\end{scope}
\end{tikzpicture}
\end{center}
  \vspace*{-0.5cm}\caption{ $P_{9}(2, 2, 2, 2)$ and $P_{9}(2, 1, 5)$ }\label{figthm5.3}
\end{figure}
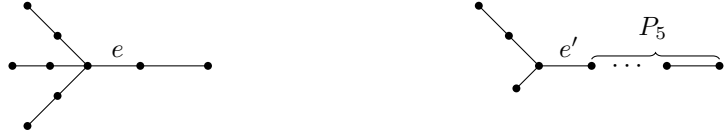

\begin{thm}\label{thm5.3}
Let $n \geq 10$. Then we have $\mathbb{E}(P_{n}(2, 2, 2, n-7))<\mathbb{E}(P_{n}(2, 1, n-4))$
\end{thm}
\begin{proof}
Consider the cut edges $e$ of $G=P_{9}(2, 2, 2, 2)$ and $e'$ of $H=P_{9}(2, 1, 5)$ as shown in Fig.\ref{figthm5.3}.

Let $G(k)$, $H(k)$ be graphs obtained by subdividing  the cut edges $e$ of $G$ and $e'$ of $H$ respectively $k$ times.
Then we have $P_{n}(2,2,2,n-7)=G(n-9)$ and $P_{n}(2,1,n-4)=H(n-9)$.
Denote $g_{k}= \widetilde{\phi}(G(k), x)$ and $h_{k}= \widetilde{\phi}(H(k), x)$.

By some directly calculations,  we have
$$\begin{aligned}
h_{0}=&\widetilde{\phi}(P_{9}(2, 1, 5)  , x) ={{x}^{9}}+8{{x}^{7}}+20{{x}^{5}}+17{{x}^{3}}+4x,\\
g_{0}=&\widetilde{\phi}(P_{9}(2, 2, 2, 2), x) ={x}^{9}+8\, {x}^{7}+18\, {x}^{5}+16\, {x}^{3}+5\, x ,\\
\end{aligned}$$
$$\begin{aligned}
h_{1}=&\widetilde{\phi}(P_{10}(2, 1, 6)  , x) ={x}^{10}+9\, {x}^{8}+27\, {x}^{6}+31\, {x}^{4}+12\, {x}^{2}+1 ,\\
g_{1}=&\widetilde{\phi}(P_{10}(2, 2, 2, 3), x) = {x}^{10}+9\, {x}^{8}+25\, {x}^{6}+28\, {x}^{4}+12\, {x}^{2}+1.
\end{aligned}$$
So we have  $ \  \ h_{1}g_{0}-h_{0}g_{1}= (2\, {x}^{4}+8\, {x}^{2}+1) ({x}^{2}+1) ^{3}> 0$ for all $x >0$.

Also we can compute that $\mathbb{E}(H(0))=\mathbb{E}(G(0))=6+2\,  \sqrt{5}$. So using Theorem \ref{thm3.1},  we have\\[0.2cm]
\hspace*{1cm}$\mathbb{E}(P_{n}(2, 1, n-4))- \mathbb{E}(P_{n}(2, 2, 2, n-7))=\mathbb{E}(H(n-9))-\mathbb{E}(G(n-9)) > \mathbb{E}(H(0))-\mathbb{E}(G(0))=0.$
\end{proof}

Notice that $P_{n}(2, 2, 2, n-7)$ and $P_{n}(2, 1, n-4)$  are quasi-order incomparable when $n\ge 11$. So Theorem \ref{thm5.3}  can not be proven by only
using the quasi-order method.
\vspace{0.5cm}

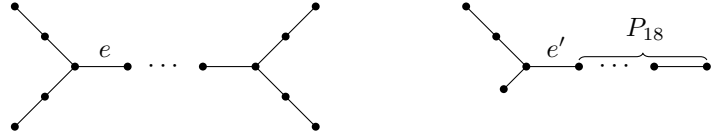
\begin{figure}[h]
\begin{center}
\begin{tikzpicture}[xshift=-3cm]
\tikzstyle myline=[line width=0.8pt]
\coordinate (A) at (-1.2, 0);
\coordinate (B) at (-0.5, 0);
\coordinate (C) at (0.5, 0);
\coordinate (D) at (1.2, 0);
\coordinate (E1) at (-1.6, 0.4);
\coordinate (F1) at (-2, 0.8);
\coordinate (E2) at (1.6, 0.4);
\coordinate (F2) at (2, 0.8);
\coordinate (E3) at (1.6, -0.4);
\coordinate (F3) at (2, -0.8);
\coordinate (E4) at (-1.6, -0.4);
\coordinate (F4) at (-2, -0.8);
\draw (A)--(B);
\draw (C)--(D);
\draw (A)--(E1);
\draw (A)--(E4);
\draw (D)--(E2);
\draw (D)--(E3);
\draw (E1)--(F1);
\draw (E2)--(F2);
\draw (E3)--(F3);
\draw (E4)--(F4);
\node at (0, 0) {$\cdots$};
\node[above] at (-0.8, 0 ) {$e$};
\foreach \point in {A, B, C, D, E1, F1, E2, F2, E3, F3, E4, F4}
{\fill [black] (\point) circle (1.5pt);}
\begin{scope}[xshift=6cm]
\tikzstyle myline=[line width=0.8pt]
\coordinate (A) at (-1.2, 0);
\coordinate (B) at (-0.5, 0);
\coordinate (C) at (0.5, 0);
\coordinate (D) at (1.2, 0);
\coordinate (E1) at (-1.6, 0.4);
\coordinate (F1) at (-2, 0.8);
\coordinate (E4) at (-1.5, -0.3);
\draw (A)--(B);
\draw (C)--(D);
\draw (A)--(E1);
\draw (A)--(E4);
\draw (E1)--(F1);
\node at (0, 0) {$\cdots$};
\node[above] at (-0.8, 0 ) {$e'$};
\draw [decoration={brace, raise=2.5pt},  decorate] (B) -- (D);
\node [right] at (0, 0.5) {$P_{18}$};
\foreach \point in {A, B, C, D, E1, F1, E4}
{\fill [black] (\point) circle (1.5pt);}
\node at (0, -1.2) {};
\end{scope}

\end{tikzpicture}
\end{center}
  \vspace*{-1cm}\caption{$T_{22}(2, 2|2, 2)$ and $P_{22}(2, 1, 18)$}\label{figthm5.4}
\end{figure}

\begin{thm}\label{thm5.4}
 Let $n \geq 22$. Then we have $\mathbb{E}(T_{n}(2, 2|2, 2))<\mathbb{E}(P_{n}(2, 1, n-4))$.
\end{thm}
\begin{proof}
Consider the cut edges $e$ of $G=T_{22}(2, 2|2, 2)$ and $e'$ of $H=P_{22}(2, 1, 18)$ as shown in Fig.\ref{figthm5.4}.

Let $G(k)$, $H(k)$ be graphs obtained by subdividing  the cut edges $e$ of $G$ and $e'$ of $H$ respectively $k$ times.
Then we have $T_{n}(2,2|2,2)=G(n-22)$ and $P_{n}(2,1,n-4)=H(n-22)$.
Denote $g_{k}= \widetilde{\phi}(G(k), x)$ and $h_{k}= \widetilde{\phi}(H(k), x)$.

 By some directly calculations,  we have
$$\begin{aligned}
h_{0}={x}^{22}& +21\, {x}^{20}+189\, {x}^{18}+953\, {x}^{16}+2955\, {x}^{14}+5824\, {x}^{12}+7293\, {x}^{10}+5643\, {x}^{8}+2541\, {x}^{6}+595\, {x}^{4}+57\, {x}^{2}+1,\\
g_{0}={x}^{22}&+21\, {x}^{20}+188\, {x}^{18}+939\, {x}^{16}+2879\, {x}^{14}+5625\, {x}^{12}+7046\, {x}^{10}+5546\, {x}^{8}+2598\, {x}^{6}+644\, {x}^{4}+64\, {x}^{2}+1,\\
h_{1}={x}^{23}& +22\, {x}^{21}+209\, {x}^{19}+1123\, {x}^{17}+3756\, {x}^{15}+8113\, {x}^{13}+11375\, {x}^{11}+10153\, {x}^{9}+5511\, {x}^{7}+1672\, {x}^{5}\\
&+241\, {x}^{3}+11\, x,\\
g_{1}={x}^{23}& +22\, {x}^{21}+208\, {x}^{19}+1108\, {x}^{17}+3667\, {x}^{15}+7850\, {x}^{13}+10982\, {x}^{11}+9912\, {x}^{9}+5546\, {x}^{7}+1768\, {x}^{5}\\
&+268\, {x}^{3}+12\, x.
\end{aligned}$$
So we have \\
\hspace*{3cm}$\begin{aligned}
&h_{1}g_{0}-h_{0}g_{1}= x({x}^{8}+7\, {x}^{6}+11\, {x}^{4}-4\, {x}^{2}-1) ({x}^{2}+1) ^{3}\\
&D=\{x|h_{1}g_{0}- h_{0}g_{1} < 0,  x>0\}\doteq (0,  0.663073 ).
\end{aligned}$

\noindent By using computer we can also find\\
$\mathbb{E}(H(0))\doteq 27.182092$, \quad   $\mathbb{E}(G(0))\doteq 27.175139$, \ and \  $\displaystyle\mathbb{E}(H(0))-\mathbb{E}(G(0))+\frac{2}{\pi}\int\limits_{{D}}\ln(\frac{h_{1}g_{0}}{h_{0}g_{1}})\text{d}x \doteq 0.000425>0.$

\noindent So by using Theorem \ref{engergycompare},  we have $\mathbb{E}(P_{n}(2, 1, n-4))-\mathbb{E}(T_{n}(2, 2|2, 2))=\mathbb{E}(H(n-22))-\mathbb{E}(G(n-22)> 0$.
\end{proof}

\begin{figure}[h]
\begin{center}
\begin{tikzpicture}[xshift=-3cm]
\begin{scope}
\tikzstyle myline=[line width=0.8pt]
\coordinate (A) at (-1.2, 0);
\coordinate (B) at (-0.5, 0);
\coordinate (C) at (0.5, 0);
\coordinate (D) at (1.2, 0);
\coordinate (E1) at (-1.6, 0.4);
\coordinate (F1) at (-2, 0.8);
\coordinate (E4) at (-1.6, -0.4);
\coordinate (F4) at (-2, -0.8);
\coordinate (E3) at (-2.4, -1.2);
\coordinate (F3) at (-2.8, -1.6);
\coordinate (E5) at (-2.4, 1.2);
\coordinate (F5) at (-2.8, 1.6);
\draw (A)--(B);
\draw (C)--(D);
\draw (A)--(E1);
\draw (A)--(E4);
\draw (E1)--(F1)--(E5)--(F5);
\draw (E4)--(F4);
\draw (E3)--(F4);
\draw (E3)--(F3);
\node at (0, 0) {$\cdots$};
\node[above] at (-0.8, 0 ) {$e$};
\draw [decoration={brace, raise=2.5pt},  decorate] (B) -- (D);
\node [right] at (0, 0.5) {$P_{22}$};
\foreach \point in {A, B, C, D, E1, F1, E4, F4, E3, F5, E5, F4, F3}
{\fill [black] (\point) circle (1.5pt);}

\end{scope}
\begin{scope}[xshift=6cm]
\tikzstyle myline=[line width=0.8pt]
\coordinate (A) at (-1.2, 0);
\coordinate (B) at (-0.5, 0);
\coordinate (C) at (0.5, 0);
\coordinate (D) at (1.2, 0);
\coordinate (E1) at (-1.6, 0.4);
\coordinate (F1) at (-2, 0.8);
\coordinate (E4) at (-1.5, -0.3);
\coordinate (F4) at (-1.8, -0.6);
\coordinate (E3) at (-2.0, -0.8);
\coordinate (F3) at (-2.4, -1.2);
\draw (A)--(B);
\draw (C)--(D);
\draw (A)--(E1);
\draw (A)--(E4);
\draw (E1)--(F1);
\draw (E3)--(F3);
\node at (0, 0) {$\cdots$};
\node[rotate=-135] at (F4) {$\cdots$};
\node[above] at (-0.8, 0 ) {$e'$};
\draw [decoration={brace, raise=2.5pt},  decorate] (B) -- (D);
\node [right] at (0, 0.5) {$P_{21}$};
\draw [decoration={brace, raise=2.5pt},  decorate] (E4) -- (F3);
\node [right] at (-2, -1) {$P_{7}$};
\foreach \point in {A, B, C, D, E1, F1, E4, E3, F3}
{\fill [black] (\point) circle (1.5pt);}
\end{scope}
\end{tikzpicture}
\end{center}
  \vspace*{-0.5cm}\caption{{$P_{31}(4, 4, 22)$} and  {$P_{31}(2, 7, 21)$}}\label{figthm5.5}
\end{figure}
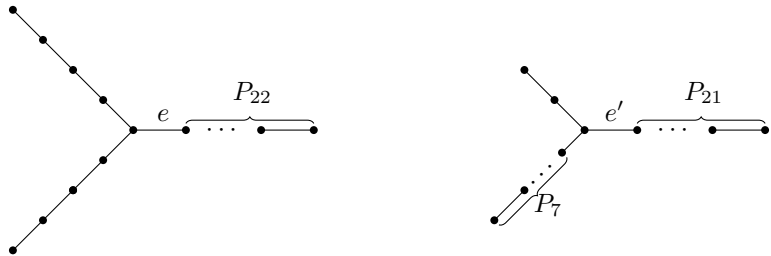

The following Theorem \ref{thm5.5} will exclude out the maximal energy tree in the class (C3) by the fourth smallest energy tree in $S_n$.

\begin{thm}\label{thm5.5}
Let $n \geq 31$. Then we have $\mathbb{E}(P_{n}(4, 4, n-9))<\mathbb{E}(P_{n}(2, 7, n-10)).$
\end{thm}
\begin{proof}
Consider the cut edges $e$ of $G=P_{31}(4, 4, 22)$ and $e'$ of $H=P_{31}(2, 7, 21)$ as shown in Fig.\ref{figthm5.5}.

Let $G(k)$, $H(k)$ be graphs obtained by subdividing the cut edges $e$ of $G$ and $e'$ of $H$ respectively $k$ times.
 Then we have $P_{n}(4,4,n-9)=G(n-31)$ and $P_{n}(2,7,n-10)=H(n-31)$.
 Denote $g_{k}= \widetilde{\phi}(G(k), x)$ and $h_{k}= \widetilde{\phi}(H(k), x)$.

By some directly calculations,  we have
$$\begin{aligned}
h_{0}=\widetilde{\phi}(P_{31}(2, 7, 21)  , x) ={x}^{31}&+30\, {x}^{29}+405\, {x}^{27}+3252\, {x}^{25}+17296\, {x}^{23}
+64220\, {x}^{21}+170943\, {x}^{19}+329768\, {x}^{17}\\+460696\, {x}^{15}
&+460851\, {x}^{13}+322620\, {x}^{11}+152131\, {x}^{9}+45426\, {x}^{7}+7738
\, {x}^{5}+619\, {x}^{3}+15\, x,\\
g_{0}=\widetilde{\phi}(P_{31}(4, 4, 22), x)
={x}^{31}&+30\, {x}^{29}+405\, {x}^{27}+3252\, {x}^{25}+17295\, {x}^{23}
+64200\, {x}^{21}+170772\, {x}^{19}+328952\, {x}^{17}\\+458317\, {x}^{15}
&+456496\, {x}^{13}+317681\, {x}^{11}+148864\, {x}^{9}+44349\, {x}^{7}+7644
\, {x}^{5}+636\, {x}^{3}+16\, x,\\
h_{1}=\widetilde{\phi}(P_{32}(2, 7, 22)  , x)
={x}^{32}&+31\, {x}^{30}+434\, {x}^{28}+3629\, {x}^{26}+20198\, {x}^{24}
+78938\, {x}^{22}+222724\, {x}^{20}+459365\, {x}^{18}\\+693530\, {x}^{16}
+&760145\, {x}^{14}+593801\, {x}^{12}+320464\, {x}^{10}+113705\, {x}^{8}+
24470\, {x}^{6}+2774\, {x}^{4}+125\, {x}^{2}+1,\\
g_{1}=\widetilde{\phi}(P_{32}(4, 4, 23), x) = {x}^{32}&+31\, {x}^{30}+434\, {x}^{28}+3629\, {x}^{26}+20197\, {x}^{24}
+78917\, {x}^{22}+222534\, {x}^{20}+458396\, {x}^{18}\\+690471\, {x}^{16}
+&753971\, {x}^{14}+585871\, {x}^{12}+314249\, {x}^{10}+111032\, {x}^{8}+
24007\, {x}^{6}+2792\, {x}^{4}+132\, {x}^{2}+1.
\end{aligned}$$
So we have

$h_{1}g_{0}-h_{0}g_{1}= x \left( {x}^{4}+3\, {x}^{2}+1 \right)  \left( {x}^{12}+12\, {x}^{10}+53
\, {x}^{8}+107\, {x}^{6}+99\, {x}^{4}+34\, {x}^{2}+1 \right)>0
$ for all $x >0.$\\
By using computer we can also find
$$\mathbb{E}(H(0))\doteq 38.616923, \qquad  \mathbb{E}(G(0))\doteq 38.616742$$
So using Theorem \ref{thm3.1},  we have  $\mathbb{E}(P_{n}(2, 7, n-10))-\mathbb{E}(P_{n}(4, 4, n-9))=\mathbb{E}(H(n-31))-\mathbb{E}(G(n-31)) \geq \mathbb{E}(H(0))-\mathbb{E}(G(0))\doteq 0.000181>0$.
\end{proof}

\begin{thm} Let $n\ge 31$. Let $S_n'=S_n\backslash \{P_n(2,5,n-8), P_n(2,3,n-6), P_n(2,1,n-4)\}$ be the first $\lfloor\frac{n-9}{2}\rfloor$ trees in the quasi-order list (\ref{equ5.1}) of $S_n$. Then $P_n$ and the $\lfloor\frac{n-9}{2}\rfloor$ trees in $S_n'$ are the first $\lfloor\frac{n-7}{2}\rfloor$ largest energy trees in the class of all trees of order $n$.
\end{thm}

\begin{proof}
It is obvious by the quasi-order list (\ref{equ5.1}) that the smallest energy tree in the set $\{P_n\}\cup S_n'$ is $P_n(2,7,n-10)$. Now take any tree $T\notin \{P_n\}\cup S_n'$ of order $n$, we consider the following four cases:

\noindent\textbf{Case }1: $T\in $(C2). Then $T\in S_n\backslash S_n'$. By the quasi-order list (\ref{equ5.1}) we have $T\prec P_n(2,7,n-10)$.

\noindent\textbf{Case }2: $T\in $(C3). Then by Theorem \ref{thm5.2} and Theorem \ref{thm5.5} we have
$$\mathbb{E}(T) \le \mathbb{E}(P_{n}(4, 4, n-9))< \mathbb{E}(P_{n}(2, 7, n-10)).$$

\noindent\textbf{Case }3: $T\in $(C4). Then by Theorem \ref{thm5.2}, \ref{thm5.3} and the list (\ref{equ5.1}) we have
$$\mathbb{E}(T) \le \mathbb{E}(P_{n}(2,2,2,n-7))< \mathbb{E}(P_{n}(2, 1, n-4))< \mathbb{E}(P_{n}(2, 7, n-10)).$$

\noindent\textbf{Case }4: $T\in $(C5).

\noindent\textbf{Subcase }4.1: $N_3(T)=2$ and $\Delta (T)=3$. Then $T$ is of the form $T_n(a,b|c,d)$. So by Theorem \ref{thm4.2}, \ref{thmtabcd}, \ref{thm5.4} and the list (\ref{equ5.1}) we have
$$\mathbb{E}(T) < \mathbb{E}(P_{n}(2, 1, n-4))< \mathbb{E}(P_{n}(2, 7, n-10)).$$

\noindent\textbf{Subcase }4.2: $N_3(T)=2$ and $\Delta (T)\ge 4$. Then a tree $T'$ with $N_3(T')=2$ and $\Delta(T')=3$ can be obtained from $T$ by using total edge grafting several times. So $T \prec T'$,  and thus by Subcase 4.1 we have $\mathbb{E}(T) < \mathbb{E}(T') < \mathbb{E}(P_{n}(2, 7, n-10))$.

\noindent\textbf{Subcase }4.3: $N_3(T)\geq 3$.  Using Theorem \ref{thmrich} several times we can obtain a tree $T'$ with $N_3(T')=2$ and $T \prec T'$. So by Subcases 4.1 and 4.2 we have   $\mathbb{E}(T) < \mathbb{E}(T') < \mathbb{E}(P_{n}(2, 7, n-10))$.
\end{proof}

\bibliographystyle{acm}

\end{document}